\documentclass[11pt]{amsart}
\usepackage{amsmath, amsthm, amssymb}    
\usepackage{graphicx,color}                    
\usepackage{hyperref}                    

\usepackage{fourier}
\usepackage[T1]{fontenc}

\newcommand{\NNN}{\mathcal{N}}

\newcommand{\bbb}{\mathbf{b}}
\newcommand{\bbn}{\mathbf{n}}

\newcommand{\bbq}{\mathbf{q}}
\newcommand{\bbr}{\mathbf{r}}
\newcommand{\bbt}{\mathbf{t}}
\newcommand{\bbx}{\mathbf{x}}
\newcommand{\RR}{\mathbb{R}}
\newcommand{\real}{\mathbb{R}}

\newcommand{\SSS}{\mathbb{S}}
\newcommand{\pd}{\partial}
\newcommand{\dd}{\mathrm{d}}

\newcommand{\beq}{\begin{equation}}
\newcommand{\eeq}{\end{equation}}

   \newtheorem{theorem}{Theorem}
   \newtheorem{proposition}[theorem]{Proposition}
	  
   \newtheorem{corollary}[theorem]{Corollary}
   \newtheorem{lemma}[theorem]{Lemma}
   \newtheorem{definition}[theorem]{Definition}

 \theoremstyle{remark}
   \newtheorem{example}[theorem]{Example}

\begin{document}
\title{Monge surfaces and planar geodesic foliations}
\author{David Brander}
\address{Department of Applied Mathematics and Computer Science,\\
Technical University of Denmark\\
 Matematiktorvet, Building 303 B\\
DK-2800 Kgs. Lyngby\\ Denmark}
\email{dbra@dtu.dk}

\author{Jens Gravesen}
\address{Department of Applied Mathematics and Computer Science,\\
Technical University of Denmark\\
 Matematiktorvet, Building 303 B\\
DK-2800 Kgs. Lyngby\\ Denmark}
\email{jgra@dtu.dk}
\date{}				


\begin{abstract}
A Monge surface is a surface obtained by sweeping  a generating plane curve 
along a trajectory that is orthogonal to the moving plane containing the curve. 
Locally, they are characterized as being foliated by a family of planar geodesic lines of curvature.
We call surfaces with the latter property PGF surfaces, and investigate the global properties
of these two naturally defined objects.
The only compact orientable PGF surfaces are tori;  these are globally Monge surfaces, 
and they have a simple characterization in terms of the directrix. 
We show how to produce many examples of Monge tori and Klein bottles,
as well as tori that do not have a closed directrix.
\end{abstract}

\maketitle

\section{Introduction}
A \emph{Monge surface} is the surface swept out by the motion of a plane curve 
(the \emph{profile curve} or \emph{generatrix}) whilst the plane is moved through space in such a way that the movement of the plane is always in the direction of the normal to the plane. 
They are also characterized as surfaces swept out by \emph{parallels},
each of which is the orthogonal trajectory of a point on the initial plane curve, and
all of which are parallel to each other.
These surfaces were defined by Monge  (\cite{monge}, 	\textsection XXV) and classically studied by,
e.g. Darboux (\cite{darboux}, Section 85), and L. Raffy \cite{raffy}.  Each meridian (i.e., the intersection of the surface with the moving plane)
is a normal section of the surface, and therefore a geodesic. Moreover, the meridians are
all congruent. Hence a Monge surface
is foliated by (congruent) planar geodesics. 

 Conversely, various practical applications, such as surfaces made from strips of wood \cite{bridges2017},
are modelled by surfaces foliated by planar geodesic lines of curvature.  We call a surface with such a foliation
 a \emph{PGF surface}.  It was known classically (see \cite{eisenhart}, Section 129) that these 
surfaces are \emph{locally} the same objects as Monge surfaces.  The purpose of this note is to consider the 
difference between the global objects given by these two natural definitions, and to investigate the
problem of constructing compact examples.

To illustrate the difference at a global level
 between Monge surfaces and PGF surfaces, it is not hard to construct examples of  
PGF surfaces that are certainly not Monge surfaces (Figure \ref{fig2}). On the other
hand, since a Monge surface is naturally defined by the arbitrary orthogonal motion of a plane
through space, singularities naturally arise.  For example, a sphere can be represented as a Monge
surface with singularities at the poles, but a sphere does not have a global foliation by planar geodesics.

 We investigate in a modern framework the basics of both the local and global 
theory that follows directly from the two natural definitions, focusing on PGF surfaces in 
Section \ref{pgfsection} and  Monge surfaces in Section \ref{mongesection}. 

In the last section we study the construction of \emph{Monge tori}. We define a Monge torus to
be a Monge surface that has the property that the profile curve and all of the parallels are closed curves.
We show how to construct many examples of Monge tori and Klein bottles.
We also answer a natural question that arises in the global theory of PGF surfaces:
in Section \ref{pgfsection} we observe (Proposition \ref{toriprop}) that if $\bbx: M \to \real^3$
is a PGF immersion, and $M$ is compact, then $M$ is necessarily a torus.  We also show (Theorem \ref{thm2}) that any 
complete immersed PGF surface $S \subset \real^3$ has a natural covering $\real^2 \to S$ consisting of
lines of curvature, one family of which is the planar geodesic foliation, the other 
family being parallels of a Monge surface parameterization.  If $S \subset \real^3$ is compact,
one can ask whether or not the leaves of this double foliation are closed. 
It is easy to see that the planar geodesics are necessarily closed.  However, one can show that
the parallels need not be closed by constructing examples of Monge surfaces that are tori
 but not ``Monge tori'' -- i.e., the parallels are not closed curves.

\section{Planar Geodesic Foliations}  \label{pgfsection}
If a plane curve is not a straight line, then the plane that it lies in is unique, and
so a family of typical planar curves gives us a natural family of planes in $\RR^3$.
For the case of a straight line, there are infinitely many planes that contain it.
In order that there be a canonical choice, consider this property of the plane
containing a non-rectilinear planar geodesic: the plane is orthogonal to the surface. 
This plane, determined by the tangent to the curve and the surface normal,
is unique;  and  a straight line in a surface lies in such a plane
if and only if the normal to the surface is constant along the line.  
\begin{definition} \label{defn1}
 A \emph{planar geodesic foliation} of an immersed surface $S \subset \real^3$ is a $1$-dimensional
foliation, the leaves of which are planar geodesic curves in $S$, with the additional property that,
 along each of the leaves, the normal to the surface is
parallel to the plane containing the curve.
\end{definition}
Note: it is not difficult to show that the conditions given on the leaves of the foliation are 
equivalent to requiring that they be planar geodesic \emph{lines of curvature}.
The above definition includes all foliations by non-rectilinear planar geodesics.  If the geodesics
are straight lines, then the surface is developable, because the normal is required to
be constant along the rulings.  We will call a surface together with a planar geodesic foliation
a \emph{planar geodesically foliated (PGF) surface}.

\subsection{Local planar geodesic foliations}
Locally, a PGF surface is just a piece of a Monge surface:
\begin{theorem}
  \label{thm:geodesic}
 A  planar geodesically foliated surface can, on a neighbourhood of any point,
 be locally parameterized in the form
  \begin{equation} \label{mseqn}
    \bbx(u,v)=\bbr(v)+x(u)\,\bbq_1(v)+y(u)\,\bbq_2(v)\,,
  \end{equation}
  with $\bbr$  a space curve with rotation minimizing frame
  $(\bbt=\bbr^\prime, \bbq_1,\bbq_2)$, and each iso-curve $\gamma_{v}(u) =  \bbx(u,v)$
	 a unit speed geodesic.
\end{theorem}
\begin{proof}
Without loss of generality, we can assume the surface is (locally)  parameterized 
as $(u,v)\mapsto\bbx(u,v)$, where the iso-curve
$\gamma_{v_0}(u) := \bbx(u,v_0)$ is a planar unit-speed geodesic for each fixed $v_0$.
Now, varying $v$
 gives us a one parameter family $\Pi(v)$ of planes in $\RR^3$. 
Let $\NNN(v)$ be a family of normals to these planes.
By integrating 
$\bbr^\prime(v) = \NNN(v)$,  with $\bbr(v_0) = \bbx(u_0, v_0)$,
  we obtain a curve $\bbr (v)$, the \emph{spine curve} (or \emph{directrix}),
	that intersects the planes orthogonally.
		Given any differentiable curve $\bbr$
and any pair of vectors $Q_1$, $Q_2$ orthogonal to $\bbr^\prime(v_0)$, there
is a unique orthonormal frame field (the \emph{rotation minimizing frame}, or \emph{relatively parallel adapted frame}, \cite{bishop})
 $(\bbt(v),\bbq_1(v),\bbq_2(v))$, where $\bbt=\bbr^\prime/|\bbr^\prime|$,
$(\bbq_1, \bbq_2)(v_0)=(Q_1, Q_2)$, and such that the derivatives of $\bbq_i$
 are parallel to $\bbr^\prime$, i.e., there is no rotation in the plane perpendicular to the spine.	
	Let $(\bbt, \bbq_1, \bbq_2)$ be
	the rotation minimizing frame along $\bbr$, with some choice of $Q_1$, $Q_2$.
 Since each curve
$\gamma_{v}$ lies in the plane through $\bbr(v)$ spanned by $\bbq_i(v)$,
we can write
\begin{equation} \label{eq:gamp}
  \bbx(u,v)=\bbr(v) + x(u,v)\,\bbq_1(v) + y(u,v)\,\bbq_2(v)\,.
\end{equation}
 The tangent to each $\gamma_{v}$ is
\[
\gamma_v^\prime = \frac{\pd\bbx}{\pd u} =
  \frac{\pd x}{\pd u}\,\bbq_1 + \frac{\pd y}{\pd u}\,\bbq_2\,.
\]
Set 
\begin{equation*}  
  \bbn _v  = \bbt\times\gamma_v' =
   - \frac{\pd y}{\pd u}\,\bbq_1 + \frac{\pd x}{\pd u}\,\bbq_2\,.
\end{equation*}
At non-inflectional points, $\bbn_v$ is the principal normal of the planar geodesic $\gamma_v$,
but in any case it is perpendicular to both $\bbt$ and $\gamma_v^\prime$.
By the assumption of Definition \ref{defn1}, 
the plane containing the curve $\gamma_v$ is spanned by $\gamma_v^\prime$ and
the surface normal. Since $\bbn_v$ also lies in this plane, it follows that
$\bbn _v$ is parallel to the surface normal or,
equivalently, 
\beq \label{geodcond}
\bbn_v \cdot \frac{\pd\bbx}{\pd v} =0.  
\eeq
As $\bbt$ is a unit vector field, there exist smooth functions $\kappa_1$ and $\kappa_2$ such that
\[
\bbt^\prime(v) = \kappa_1(v) \bbq_1(v) + \kappa_2(v) \bbq_2(v).
\]
The rotation minimizing
assumption means that
\[
  \frac{\pd\bbx}{\pd v}
  = (1-x\,\kappa_1-y\,\kappa_2)\,\bbt
  + \frac{\pd x}{\pd v}\,\bbq_1 + \frac{\pd y}{\pd v}\,\bbq_2\,.
\]
So the  condition \eqref{geodcond} is 
\[
\frac{\partial x}{ \partial u} \frac{\partial y}{\partial v} -\frac{ \partial y }{\partial u} 
  \frac{\partial x }{\partial v}= 0,
\]
 i.e., that the map $\phi (u,v) = (x(u,v), y(u,v))$ has rank less than two. Since, 
from \eqref{eq:gamp}, we must have $\partial \phi/\partial u \neq 0$, the image of $\phi$ is a curve.
In other words, each curve $\alpha_v(u) = (x(u,v), y(u,v))$ is a reparameterization of
a single curve $\alpha_{v_0}(u)$.  From the assumption that $\gamma_v(u)$ is a unit speed 
geodesic, it follows easily that this reparameterization $u \mapsto \tilde u(u)$ is nothing but 
a translation $\tilde u(u) = u+c(v)$, and from this the stated result follows.
\end{proof}
We will call a parameterization of the form \eqref{mseqn} a \emph{Monge surface parameterization},
as it follows from the definition of a Monge surface that a (regular) Monge surface has this form.
We mention here that, given a Monge
surface parameterization, where both the spine curve $\bbr(v)$ and 
the profile curve $(x(u),y(u))$ are parameterized
by arc-length, the metric takes the form:
\beq \label{metric}
\dd s^2 = \dd u^2 + (1- \alpha)^2 \dd v^2, \quad \quad
  \alpha(u,v) := x(u) \kappa_1(v) - y(u) \kappa_2(v).
\eeq
Furthermore, one computes that the second fundamental form is diagonal, so $(u,v)$ 
are principal coordinates, i.e., the parallels and meridians on a Monge surface are lines
of curvature.

Let us call a surface \emph{locally planar geodesically foliated} if there exists a neighbourhood
around each point that admits a Monge surface parameterization.
Using the fact that all geodesics on a sphere are planar,
it is easy to construct complete examples of such surfaces with arbitrary topology 
by glueing in a $C^\infty$ manner surfaces of revolution to pieces of spheres, as in
Figure \ref{fig1} (left), and then assembling multiple units.

\begin{figure}[h!tbp]
	\begin{center}
	$
	\begin{array}{cc}
		\includegraphics[height=35mm]{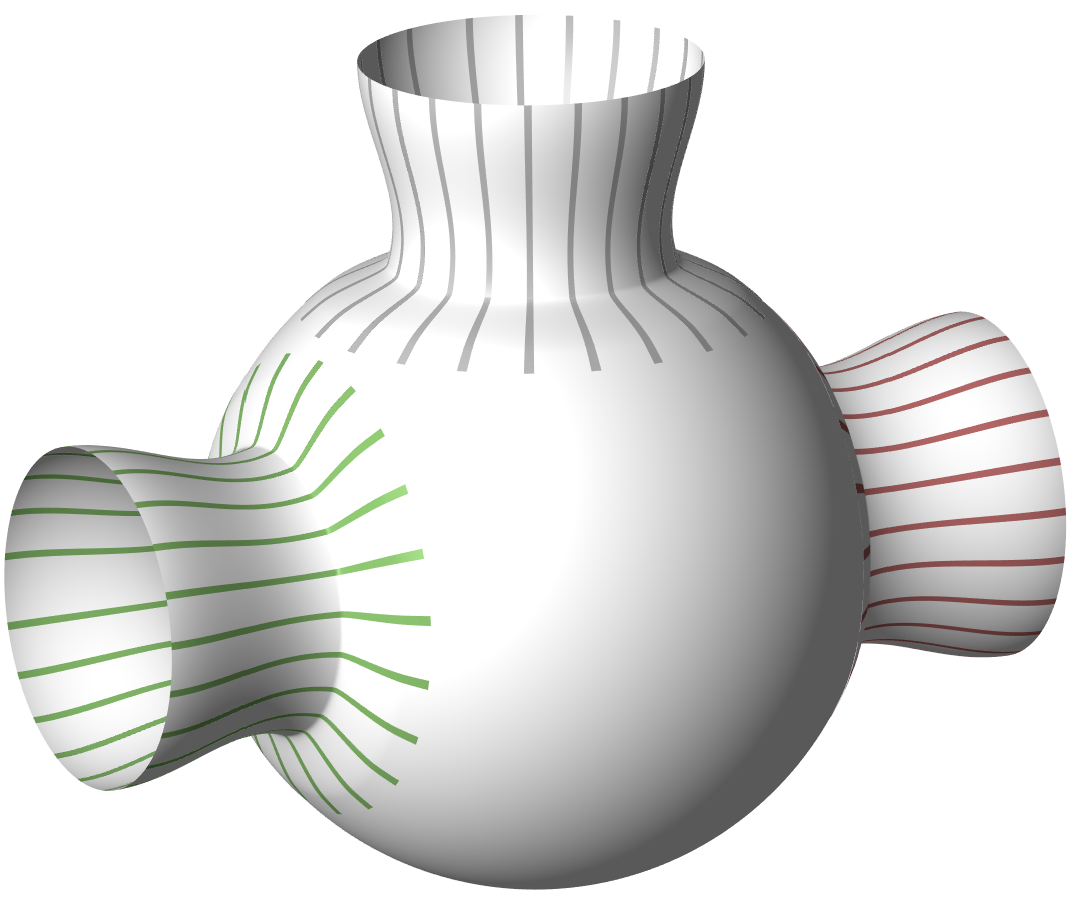}  	\quad & \quad
				\includegraphics[height=35mm]{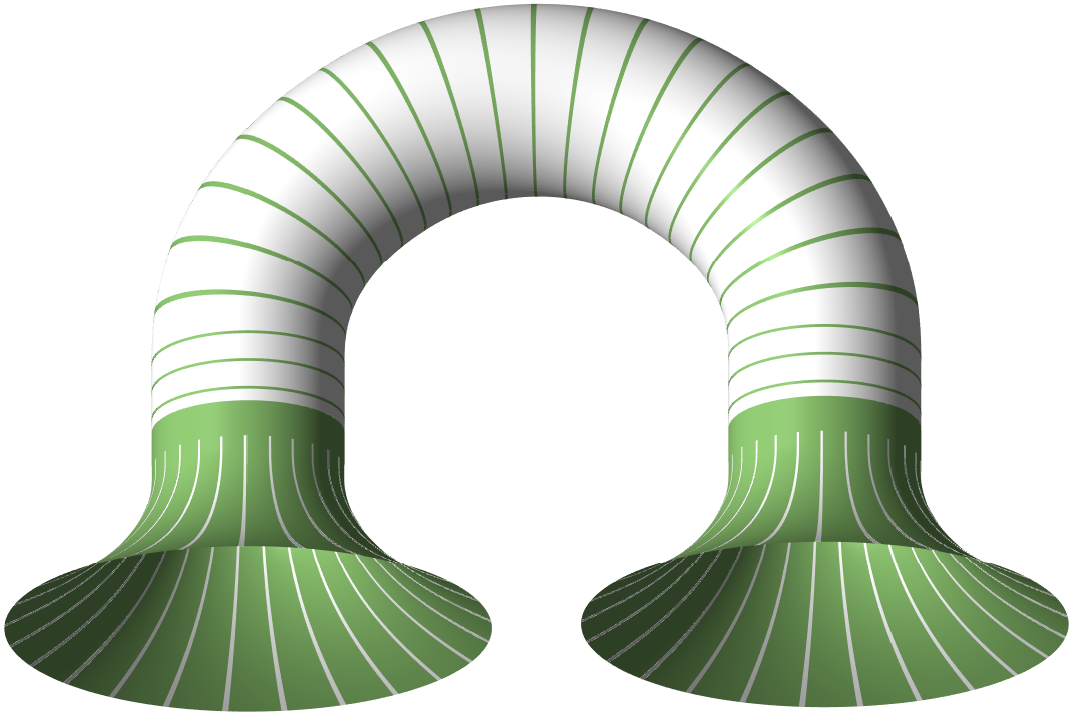}  	
\end{array}
	$
		\end{center}
	\caption{  $C^\infty$ immersed surfaces that are locally, but not globally,  PGF surfaces.}
	\label{fig1}
\end{figure}

Any surface such as those shown in Figure \ref{fig1}, where different choices 
of foliation are made for different regions of the surface, must, on
the overlap of these regions, have more than one possible planar geodesic foliation.
Obviously a piece of a plane, a sphere or a cylinder has this property.
Theorem \ref{thm:geodesic} implies that these are essentially the only possibilities:
\begin{theorem}
Let $S$ be a connected surface that admits more than one distinct 
planar geodesic foliation.  Then, either $S$ is an open subset of a sphere,
or every point of $S$ has a neighbourhood $U$ that 
is an open subset of a cylinder,  $\gamma \times \real$, where $\gamma$ is a plane curve.
\end{theorem}
\begin{proof}
Suppose $S$ has two distinct foliations by planar geodesics, say $\mathcal{F}_1$ and 
$\mathcal{F}_2$. Let  $p$ be a point of $S$, and suppose first that 
$p$ is not an umbilical point. By Theorem \ref{thm:geodesic},
there is a neighbourhood $U_1$ of $p$ such that $U_1$ is parameterized as a Monge surface
$\bbx(u,v)$, where the geodesics are given by $v=\hbox{constant}$, and an analogous 
neighbourhood $U_2$ for $\mathcal{F}_2$, parameterized by $\tilde \bbx(\tilde u, \tilde v)$,
where the planar geodesics are $\tilde v = \hbox{constant}$. 
 Set $U = U_1 \cap U_2$.
We can assume that $U$ is small enough so that it contains no 
 umbilic points. Then both $(u,v)$ and $(\tilde u, \tilde v)$ are
curvature line coordinates on $U$. Since the lines of curvature through a non-umbilic point are unique,
and  the planar geodesic foliations are assumed to be
distinct, it follows that the $u$ isocurves are the $\tilde v$ isocurves and vice versa, and we have
$(\tilde u(u,v), \tilde v(u,v)) = (\tilde u(v), \tilde v(u))$.
Moreover, all the $u$ isocurves are congruent to each other, and all the $v$ isocurves are congruent to each other.
We can assume that both sets of  coordinates  are chosen such 
that both the profile curve and the spine curve are unit speed.
From \eqref{metric} and the relation between the variables we have
 \[
\dd s^2 = \dd u^2 + (1-\alpha(u,v))^2\dd v^2 =
   (1-\tilde \alpha)^2 \left(\frac{\partial \tilde v}{\partial u}\right)^2 \dd u^2 
	+ \left(\frac{\partial \tilde u}{\partial v}\right)^2 \dd v^2,
	\]
	where $\partial \tilde u/\partial v$ depends only on $v$.  Hence $(1-\alpha)^2$ depends only on $v$,
	and is constant equal to $1$, given that the spine curve $\bbr (v)$ is unit speed. 
	Thus, $\dd s^2 = \dd u^2 + \dd v^2$,
 which implies that
the surface is developable.  A developable surface is locally a ruled surface, and, at non-umbilical
points, the rulings are
lines of curvature. Thus one of the families of planar geodesics consists of 
straight line segments. Since these line segments are also \emph{parallel}, due to the second Monge surface
representation, the surface $U$ is a  cylinder.

Now suppose that $p$ is an umbilical point where the Gaussian curvature $K(p)$ at $p$ is non-zero:
then $K$ must
be non-zero on a neighbourhood $W$ of $p$.  By the discussion above,  all of the points in $W$ must be umbilic,
because non-umbilic points have Gaussian curvature zero. Hence $W$ is part of a sphere, with constant
curvature $K_0$.  Moreover, the set of points at which the Gauss curvature is $K_0$ is clearly
both open and closed,
hence the whole of $S$.    

The remaining case is that $p$ is an umbilic and $K(p)=0$.  Since $S$ is not part of a sphere, there can 
(by the above) be
no umbilic points with positive Gaussian curvature, and hence the Gauss curvature must be zero everywhere.
Choose a neighbourhood $U$ of $p$ that has a Monge surface representation $\bbx (u,v)$. Since $K=0$,
we can assume that $U$ is a part of a ruled surface. Then either $U$ is a part of a plane and
the proof is complete, or there exists a point $q$ in $U$ which is not an umbilic point. In the latter case,
as previously discussed, the rulings, on a neighbourhood of $q$, are parallels of the Monge representation.
Since a parallel curve to a straight line segment is also a straight line segment, it follows that all of the parallels
on $U$ are parallel straight line segments, and $U$ is a cylinder.
\end{proof}

\subsection{Global planar geodesic foliations}
We now consider global planar geodesically foliated surfaces.
In general, a PGF surface does not have a global representation as a Monge surface. Figure \ref{fig2}
shows that it is easy to construct counterexamples: the developable surface to the left, although orientable,
 does not have a well-defined vector-field orthogonal to the foliation, whereas a Monge surface
does, given by the direction of the parallels.    The surface to the right has a different profile curve on 
one end of the surface from the other.
\begin{figure}[h!tbp]
	\begin{center}
	$
	\begin{array}{cc}
	\includegraphics[height=30mm]{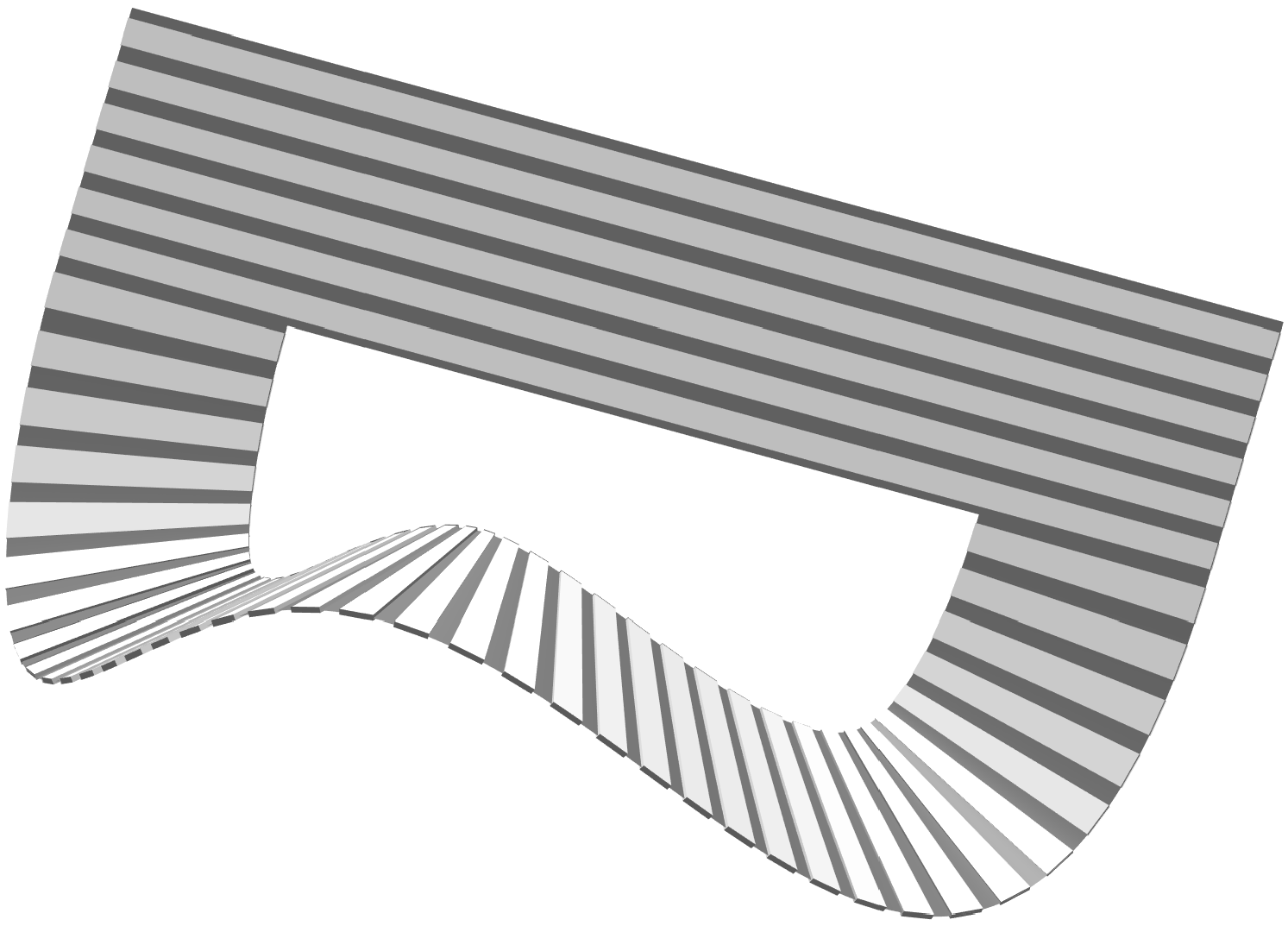}    \quad & \quad
											\includegraphics[height=30mm]{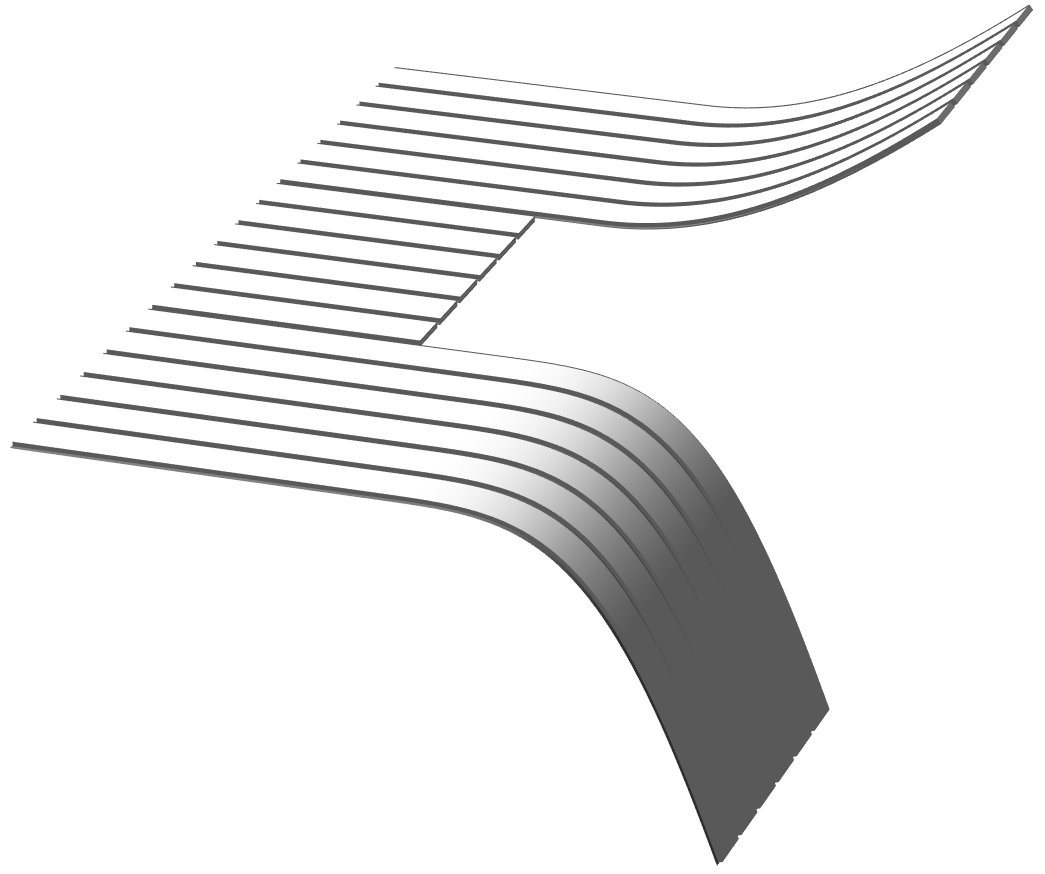}  	
\end{array}
	$
		\end{center}
	\caption{PGF surfaces that are not globally Monge surfaces..}
	\label{fig2}
\end{figure}

 Neither of the PGF surfaces in Figure \ref{fig2}  is \emph{complete}.
If we add this requirement, then such examples are ruled out:
\begin{theorem} \label{thm2}
Let $S$ be a  \emph{complete, connected, orientable} PGF surface. Then $S$ has a global parameterization
 in the form \eqref{mseqn}.
In particular, this parameterization is a covering map $\real^2 \to S$ consisting of curvature line coordinates.
\end{theorem}
\begin{proof}
Let $\gamma_0$ be a choice of one of the planar geodesics in the foliation, with a given arc-length
parameterization. Denote the surface normal by $N$ and
 let $\nu=N \times \gamma_0^\prime$ denote the unit conormal along $\gamma_0$.
Such a conormal vector field
is defined along any of the geodesics in the foliation, and we can always assume that, on neighbouring
geodesics, the  parameterizations are chosen so that $\nu$ varies continuously.
Hence, given a point $p_0 = \gamma_0(s_0)$, we can integrate the vector field $\nu$ to obtain a curve $\bbr$
in $S$, orthogonal to the foliation.
 Since $S$ is complete and $\nu$ has unit length, $\bbr$ is well-defined $\real \to S$. 
Let $\tilde S$ be the Monge surface constructed from $\bbr$ and $\gamma_0$ by the equation
 \eqref{mseqn}.  This Monge surface is regular because a Monge surface is always regular on an open
set containing the spine curve, and the completeness of $S$, together with uniquenss of geodesics,
means that $\tilde S$ is a subset of
the regular surface $S$.
Now $\tilde S$ is open in $S$, because any point on $\tilde S$ has a neighbourhood
parameterized in the form \eqref{mseqn} with
 $(u,v) \in (u_0-\varepsilon, u_0 + \varepsilon) \times (v_0-\delta, v_0+\delta)$.
To show that $\tilde S$ is the whole of $S$, consider an arbitrary point
$q \in S$.  Choosing the planar geodesic of the foliation that passes through $q$, we can 
similarly construct another Monge surface $\hat S$ that contains $q$. If $\hat S$ intersects $\tilde S$,
then there is a planar geodesic that is contained in both $\hat S$ and $\tilde S$. In this
case, it follows by the construction of these sets that $\hat S =\tilde S$. 
Thus $\hat S$ and $\tilde S$ are either equal or disjoint. But both sets are open 
and  $S$ is connected, so they must be equal, and $q \in S$.
By construction, the representation $\tilde S$ is a regular  immersion, hence a 
covering map $\real^2 \to S$.
\end{proof}
Since the $(u,v)$ coordinates of the covering map in Theorem \ref{thm2} are canonically defined
by the planar geodesic foliation,
up to coordinate changes of the form $(u,v) \mapsto (\tilde u(u), \tilde v(v))$, 
 this can be regarded as a canonical covering
by the conformal Lorentz plane $\real^{1,1}$. Hence:
\begin{corollary}
A complete, connected, orientable PGF surface  carries a Lorentzian conformal structure,
the null lines of which are given by the planar geodesics and the family of parallels of a Monge
surface parameterization.
\end{corollary}
Note that not all planar geodesic foliations carry a global Lorentz structure (see
Figure \ref{fig2}, left).

The Poincar\'e index theorem states that on a compact surface the Euler characteristic is
equal to the sum of the indices of the zeros of any vector field.  Since any foliation gives
a natural non-vanishing vector field, namely the tangent field to the foliation, this implies:
\begin{proposition}  \label{toriprop}
If $M$ is a compact surface and $\bbx : M \to \real^3$ a PGF immersion, then $M$ is a torus.
\end{proposition}

\section{Generalized Monge surfaces}  \label{mongesection}
We have shown that a PGF surface is not necessarily a Monge surface, unless it is complete.
In this section we explore the converse direction,  in order to
clarify the precise difference between Monge surfaces and PGF surfaces,
and find global conditions under which these objects coincide.

The idea of a Monge surface is defined kinematically by a plane $\Pi(t)$ moving through space
such that the movement of any point on the surface is always orthogonal to the plane.
We can describe such a moving plane by making an initial choice $Q_1$, $Q_2$ of orthonormal
basis for the plane, and  setting $\bbq_i(t)$ to be the corresponding vectors at time $t$. Similarly,
making an initial choice of point $p_0$ on the plane, we let $p(t)$ be the corresponding point at time $t$. 
Finally, we set $\bbt(t)$ to be the normal to the plane at time $t$.  By the 
definition of $\Pi(t)$, the derivatives of $p$, $\bbq_1$
and $\bbq_2$ have components only in the direction of $\bbt$. Thus we have a map $F: (a,b) \to ASO(3)$,
from some open interval into the group of orientation preserving Euclidean motions, given by
\[
F = \left( \begin{array} {cccc} \bbq_1 & \bbq_2 & \bbt & p \\
           0 & 0 & 0 & 1 \end{array} \right),
\]
and the conditions on the derivatives are represented by
\[
\omega = F^{-1} \frac{\dd F}{\dd t} =  \left( \begin{array} {cccc} 
  0 & 0 & -\kappa_1 & 0 \\
	0 & 0 & -\kappa_2 & 0 \\
	\kappa_1 & \kappa_2 & 0 & \lambda \\
	0 & 0 & 0 & 0 \end{array} \right).
\]
We call the family of planes $\Pi(t)$, defined by $F$ an \emph{orthogonal family of planes},
and say that the family is \emph{regular} if $(\kappa_1(t), \kappa_2(t), \lambda(t)) \neq (0,0,0)$, i.e., $\omega \neq 0$,
 for all $t \in (a,b)$.
This definition of regularity does not depend on the 
 choice of initial vectors $Q_1$ and $Q_2$ or the initial point $p_0$. A different choice amounts to post-multiplying the 
frame $F$ by a constant invertible matrix , and thus $\omega$ is conjugated by a constant invertible matrix,
which does not change the condition $\omega \neq 0$.

\begin{example}  \label{example1}
Given an arbitrary regular space curve $\bbr$, a regular orthogonal family of planes can be constructed by
$p(t)=\bbr(t)$, $\bbt = \bbr^\prime/|\bbr^\prime|$ and $\bbq_1$, $\bbq_2$ a rotation minimizing frame.
In that case $\lambda$ is non-vanishing.  More generally, there exist
regular orthogonal families of planes where  $\lambda$ vanishes, regardless of the choice of
initial point $p_0$. For example, take 
\[
\bbq_1(t) = (\cos t, \sin t, 0), \quad \bbq_2(t) = (0,0,1), \quad \bbt(t) = (-\sin t, \cos t, 0), \quad p(t) = \bbt(t)-t \bbt^\prime(t).
\]
This is a regular orthogonal family of planes, with $\kappa_1(t)=-1$, $\kappa_2=0$ and $\lambda(t) = t$.
Here $\lambda(t)$ vanishes at $t=0$. If we choose a different initial point $\tilde p_0$ on the plane, 
we have $\tilde p(t) = p(t) + c_1 \bbq_1(t) + c_2 \bbq_2(t)$ for some constants $c_1$ and $c_2$. Then 
$\tilde p^\prime(t) = (t+c_1) \bbt(t)$, so $\lambda(t)$ vanishes at $t=-c_1$.
\end{example} 

\begin{definition}
Let $\gamma:(a,b) \to \real^2$ be a curve parameterized by arc-length, $\gamma(s) = (x(s),y(s))$,
  and let $F: (c,d) \to ASO(3)$ be a regular 
orthogonal family of planes, with notation as above. Then the map $\bbx: (a,b) \times (c,d) \to \real^3$ given by
\[
\bbx(u,v) = p(v) + x(u) \bbq_1(v) + y(u) \bbq_2(v),
\]
is called a \emph{(generalized) Monge surface}.
\end{definition}
It should be noted that, if the initial point $p_0$ and the initial frame $(Q_1, Q_2)$ are changed in the representation
$F$ for the family $\Pi(t)$, then the curve $\gamma$ would need to be translated and rotated accordingly to 
get the same surface.

Since
\[
\frac{\partial \bbx}{\partial u} = x^\prime(u) \bbq_1(v) + y^\prime(u) \bbq_2(v), \quad \quad
\frac{ \partial \bbx}{\partial v} = (\lambda(v) - \kappa_1(v) x(u) - \kappa_2(v) y(u) )\bbt,
\]
with $(x^\prime)^2 +(y^\prime)^2=1$, a Monge surface is regular at $(u,v)$ if and only if 
\[
\lambda(v) \neq \kappa_1(v) x(u) + \kappa_2(v) y(u).
\]
Geometrically, in each plane $\Pi(v)$, there is a (possibly empty) line, the instantaneous axis of rotation of the
plane $\Pi(v)$, 
\[
L_v = \{(x,y) ~|~ \kappa_1(v) x + \kappa_2(v) y -\lambda(v) = 0 \},
\]
and the Monge surface is regular if and only if none of the lines $L_v$  intersect the
initial curve $\gamma$.  Note that the spine curve $p$ is a curve on the surface $\bbx$ if and only if $(x(u), y(u))=(0,0)$ for some
$u$.  One can always arrange that $p$ is a curve on the surface by choosing $p_0$ to be a point on the surface, and
then translating $\gamma$ accordingly.

\begin{figure}[h!tbp]
	\begin{center}
	$
	\begin{array}{ccc}
	\includegraphics[height=25mm]{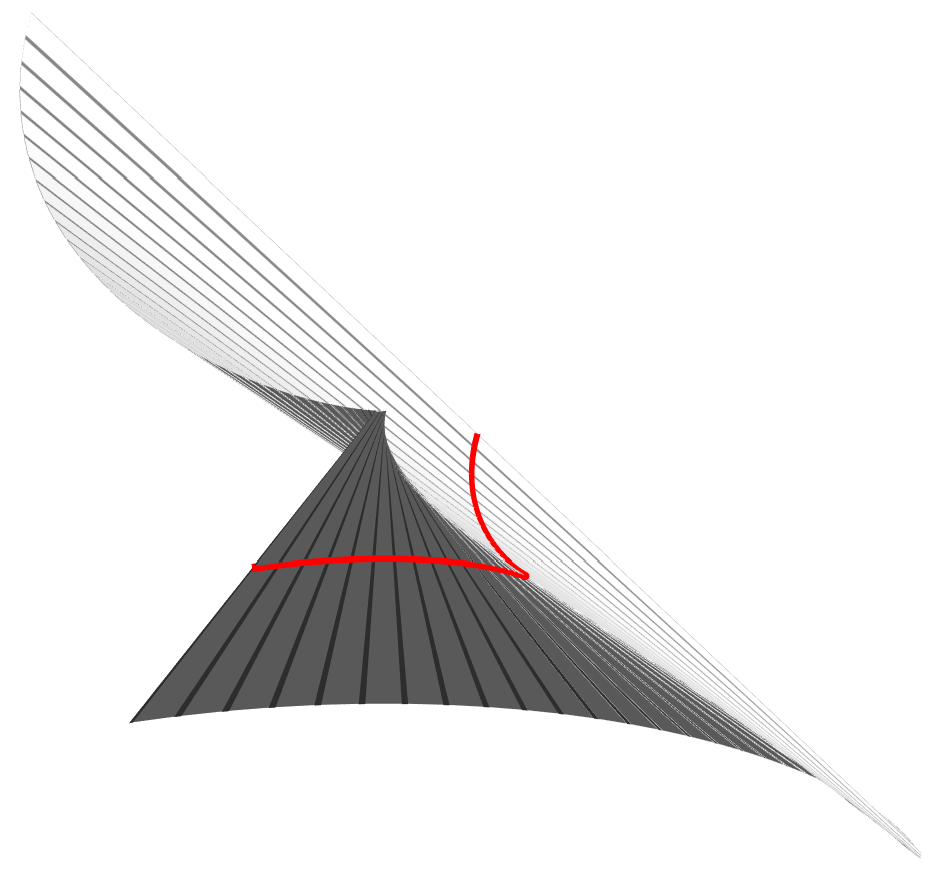}   \quad & 
											\includegraphics[height=25mm]{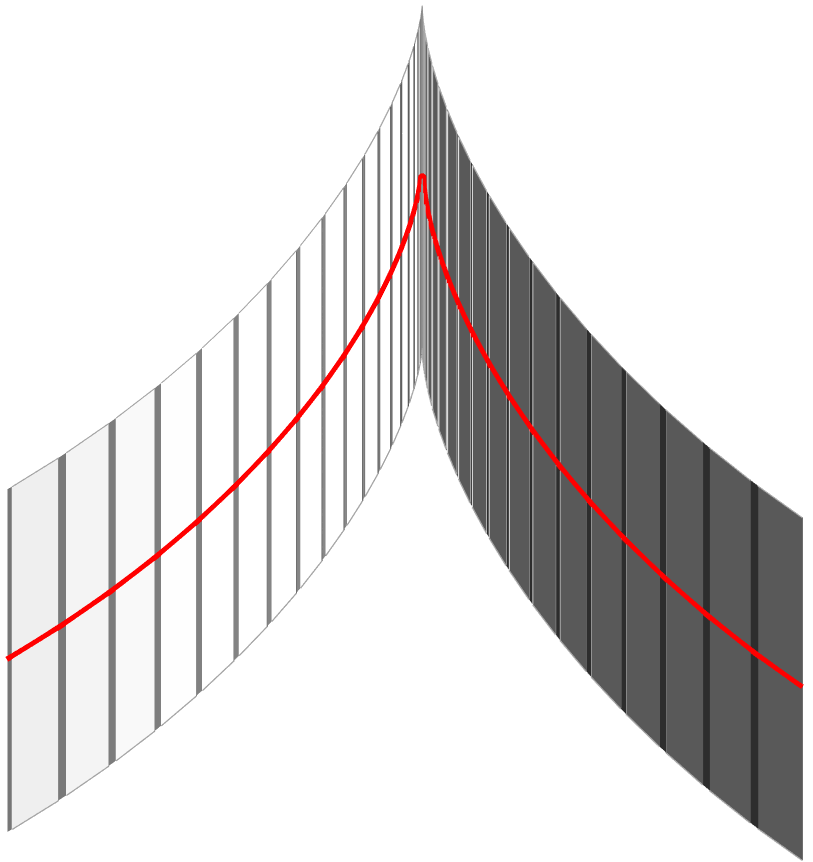}  	& \quad
											\includegraphics[height=25mm]{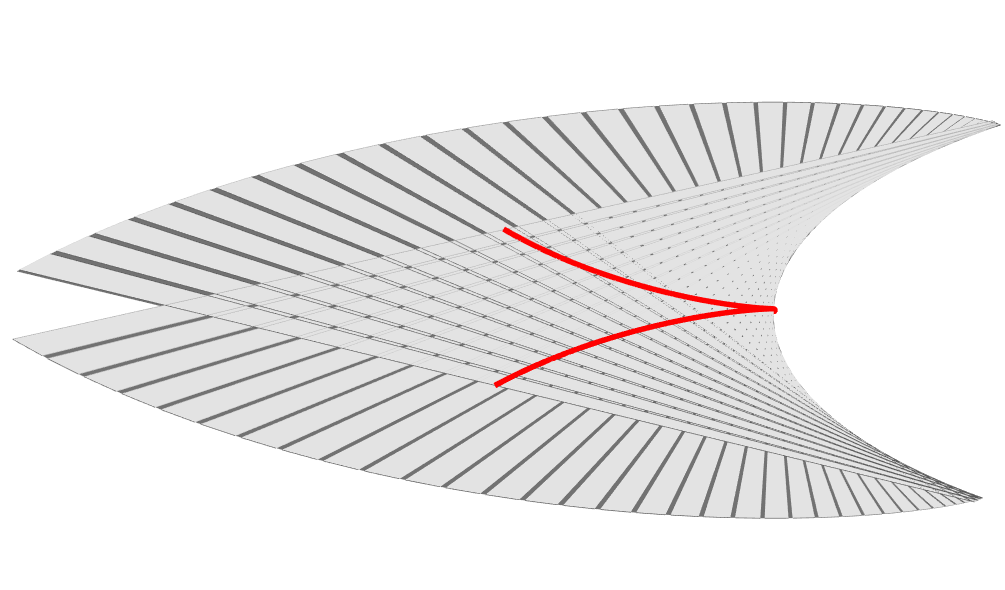}  
\end{array}
	$
		\end{center}
	\caption{Monge surfaces with the same singular spine curve, shown in red.}
	\label{fig:singularities}
\end{figure}

\begin{example} 
Returning to the family of planes in Example \ref{example1}, we have $L_v$ given by 
the equation $x=v$. That is, the instantaneous axes of rotation are all vertical lines,
and these lines  sweep out the entire $xy$-plane as $v$ runs over $\real$.  Hence there is no choice of profile curve $\gamma$ 
such that $\bbx$ is everywhere a regular surface.
Portions of the surfaces corresponding to three different lines as profile curve (the line $y=-x$, the $x$-axis
and the $y$-axis) are shown in Figure \ref{fig:singularities}. The first two have cuspidal edges, the last a fold singularity.
\end{example}

Note that the example of the sphere shows that the image of a generalized Monge surface can  be a regular
surface even if the parameterization $\bbx$ is not regular.  In general, however,
there will  be singularities that are not removable, as in Figure \ref{fig:singularities}.

For a given generalized Monge surface, if the surface is regular then, by construction, it is a PGF surface.
Combined with  Theorem \ref{thm2}, we conclude:
\begin{theorem}
Every regular Monge surface is a PGF surface, and every complete PGF surface
has a covering by a Monge surface.
In particular, \emph{complete, regular} PGF surfaces are the same objects as complete regular Monge surfaces.
\end{theorem}

\section{Monge  tori and Klein bottles}
In this section we investigate the problem of constructing \emph{Monge tori}. 
These are complete, regular surfaces, so can be regarded as either Monge surfaces or PGF surfaces.
If the spine curve 
is a \emph{closed} curve, i.e., a map $\bbr: \real \to \real^3$,
satisfying  $\bbr(v+P) = \bbr(v)$ for some $P$,
it does not follow that the surface closes up to form a topological cylinder.
Apart from very special cases such as a \emph{tubular surface} (Figure \ref{fig:tori}, right),  it is
necessary that the
surface normal, and hence the rotation minimizing frame, be periodic with the same period.
This is a property of the geometry of the spine curve, and it has a very simple characterization for curves
that have a well-defined Frenet frame. 

\subsection{Curves with closed rotation minimizing frames and Monge tori}
 Throughout this section, by a differentiable curve $\bbr: I \to \real^3$,
we mean a regular unit-speed curve that has a well-defined differential binormal field
$\bbb$ along the curve.  The principal normal is then defined by 
the formula $\bbn = \bbb \times \bbt$,
where $\bbt= \bbr^\prime(s)$.  The signed curvature is 
given by $\bbt^\prime(s) = \kappa(s) \bbn(s)$, and the torsion by $\bbb^\prime(s) = \tau(s) \bbn(s)$.
  Given a closed curve with period $L$, denote the total torsion
by $T= \int_0^L \tau(s) \dd s$. Note that the assumed binormal field
is not automatically periodic: for example a closed curve could contain pieces of 
line segments, along which the binormal field could be anything at all. We therefore
add the assumption of periodicity of the binormal field:
\begin{lemma}  \label{closedlemma}
Let $\bbr:\real \to \real^3$ be a closed, differentiable curve with period $L$,
and suppose that the binormal field $\bbb$ is also periodic with the same period.
Let $F=(\bbt, \bbq_1,\bbq_2)$ be any choice of rotation minimizing frame along $\bbr$. 
Then $F$ is periodic with period $L$ if and only if the total torsion $T$ is an integer 
multiple of $2\pi$.
\end{lemma}
\begin{proof}
Because $\bbt$ and $\bbb$ are periodic with the same period, 
the Frenet frame $(\bbt, \bbn, \bbb)$ also has this property.
 The torsion $\tau$ of the curve, given by
 $\bbb^\prime(s) = \tau(s) \bbn(s)$, is the rate of rotation of 
the basis $(\bbn, \bbb)$ for the plane $\Pi_s = \hbox{span}(\bbn, \bbb)$ 
compared with the zero rotation of the rotation minimizing  basis $(\bbq_1, \bbq_2)$.
The latter basis (and hence $F$) thus closes up if and only the 
$2 m \pi = \int \tau \dd s$, for some $m \in {\mathbb Z}$.
\end{proof}
Let us define a \emph{Monge cylinder} to be a regular Monge surface with closed spine curve 
and closed rotation minimizing frame.
Note that this definition does not depend on the choice of spine curve: if $\bbr$ is the spine curve, then any other spine curve is of the
form $\tilde\bbr(s)=\bbr(s)+a\bbq_1(s)+b\bbq_2(s)$, which has the same properties.
Given a Monge cylinder, it is a simple matter to construct a Monge torus by choosing as profile curve
a small closed curve $\gamma_0$ in the plane
$\Pi_0 = \hbox{span}(\bbq_1(0), \bbq_2(0))$, such that the origin lies on $\gamma_0$. 
The Monge surface so constructed is regular on an open set containing the
spine curve, which lies on the surface.
  Since the image of the spine curve is compact, it is always possible to choose  $\gamma_0$ small enough so that
the torus produced is regular.

\subsection{Constructing Monge tori}
Given the characterization in Lemma \ref{closedlemma}, it is not difficult to find examples of Monge cylinders, and hence tori.
For example, a plane curve has zero torsion, but more generally, one can show \cite{scherrer} that a surface in $\real^3$ is a part of
a sphere or a plane if and only if the  total torsion of any closed curve in the surface is zero. 
 Hence any close curve in a sphere can also be used to produce Monge cylinders.

The problem of constructing more general examples amounts to finding more general 
curves that satisfy the closing condition of Lemma \ref{closedlemma}.  We are also interested
in the case that the total torsion is an \emph{odd} multiple of $\pi$, because this
allows us to construct M\"obius strips and Klein bottles.  More generally still, since the total
torsion is the total rotation angle of the rotation minimizing frame as the curve is traversed,
by taking a profile curve $\gamma_0$ that has a rotational symmetry of order $n$ around the
origin, and using a spine curve with total torsion $2k\pi/n$ (where $k$ is an integer), the Monge surface will close up
after one period of the loop, although the Monge parameterization will be a covering of degree $d\leq n$, 
depending on the common divisors of $n$ and $k$.
If the total torsion is an irrational multiple of $\pi$, then a \emph{circle} centered at the origin
 can be used as profile curve, to produce a \emph{tubular} surface that is not a 
parameterized Monge torus,  but the image of which is a torus. Note that the spine curve 
does not lie on the surface in this case, but at the center of the tube.  Examples constructed
using the method described in the next subsection are shown in Figure \ref{fig:tori}.

\begin{figure}[h!tbp]
	\begin{center}
	$
	\begin{array}{ccc}
		\includegraphics[height=29mm]{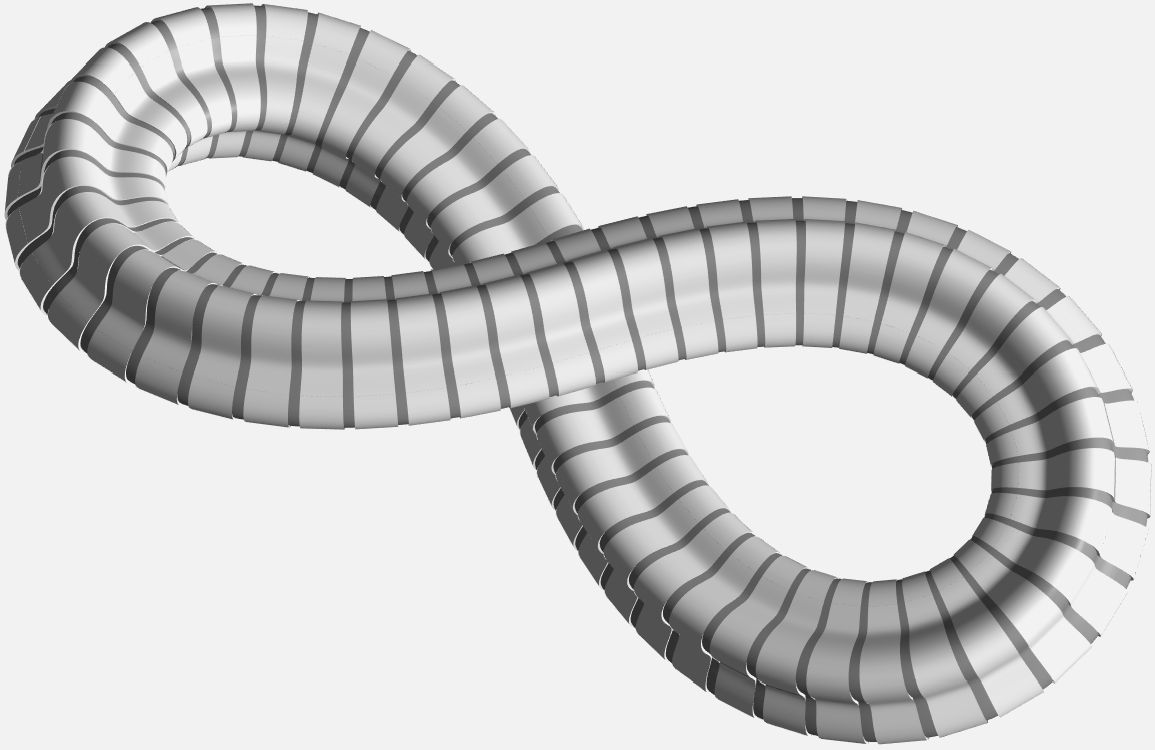}  	\quad & 
				\includegraphics[height=29mm]{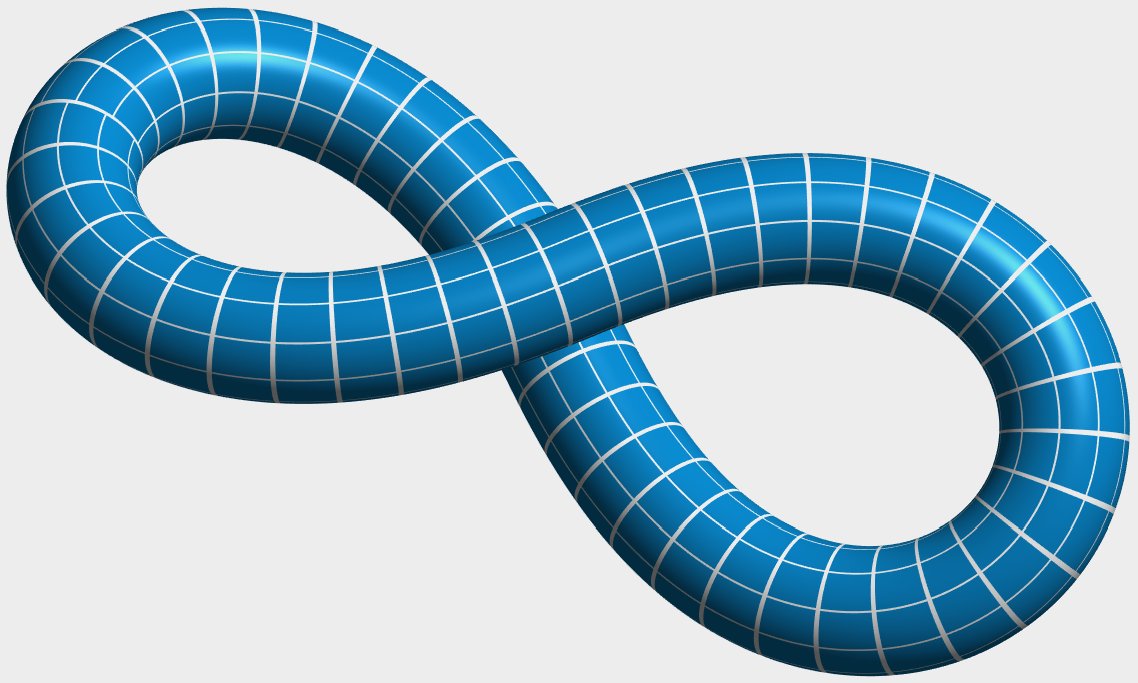}  & 
					\includegraphics[height=29mm]{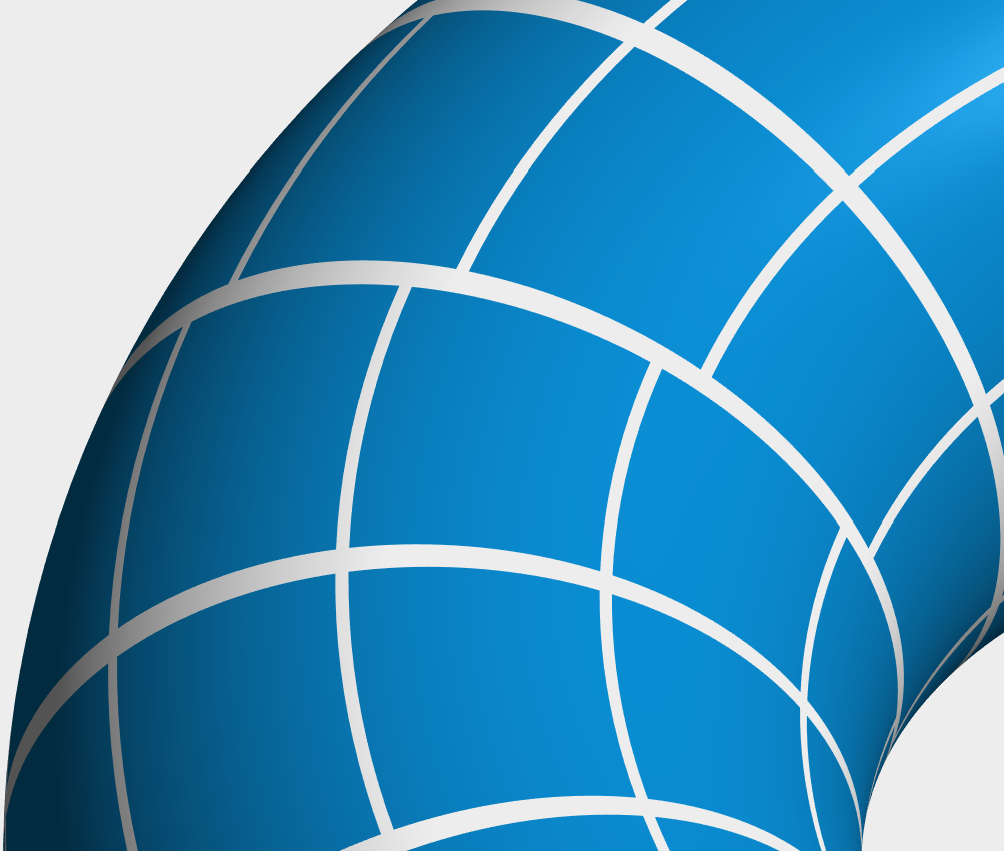}  	
\end{array}
	$
		\end{center}
	\caption{Monge coverings of embedded tori. Left: $T=L_S = 2\pi/4$, degree of covering is $4$.
	The profile curve has order $4$ symmetry about the origin of the plane.
	Right: $T=L_S = 2\pi/\pi$, degree of covering is uncountably infinite. The profile curve is a circle,
	the parallels do not close up.}
	\label{fig:tori}
\end{figure}

It is thus interesting to know whether one can find closed curves with arbitrary values for the total torsion,
or at least with arbitrary values in $[0,2\pi)$. 
The answer to this question can be found in a survey  by Fenchel \cite{fenchel}, the relevant
details of which we recall here: 
A convenient way to find curves with a specified value
for the total torsion is to look for a curve with everywhere \emph{positive}
torsion, because if $s$ is the arc-length
on a curve $\bbr$ and $s_B$ the arc-length of the binormal $\bbb$ then $\dd s_B= |\tau(s)| \dd s$.
Hence, if $\tau>0$, and $L_S$ is the length of the trace of $\bbb$ on $\SSS^2$, then
\[
T = \int _0^L \tau(s)\dd s = \int_0^{L_S} \dd S_B = L_S.
\]
In other words, if $\tau$ is non-vanishing, then
the total torsion is just the length of the trace of the binormal curve on $\SSS^2$.
Fenchel also shows that a closed curve $\bbb$ in $\SSS^2$ is the binormal to a \emph{closed} curve
$\bbr$ in $\real^3$ if and only if the great circles tangent to $\bbb$ sweep out the whole of the
$2$-sphere.  Moreover, the torsion of $\bbr$ is non-vanishing if $\bbb$ 
has no  cusps. There are many examples of such curves $\bbb$ (see Figure \ref{figcurves}), and 
 curves such as those shown can clearly be scaled so that the length is anything
 between $0$ and $2\pi$, which solves  the problem under discussion.

\begin{figure}[h!tbp]
	\begin{center}
	$
	\begin{array}{cc}
		\includegraphics[height=30mm]{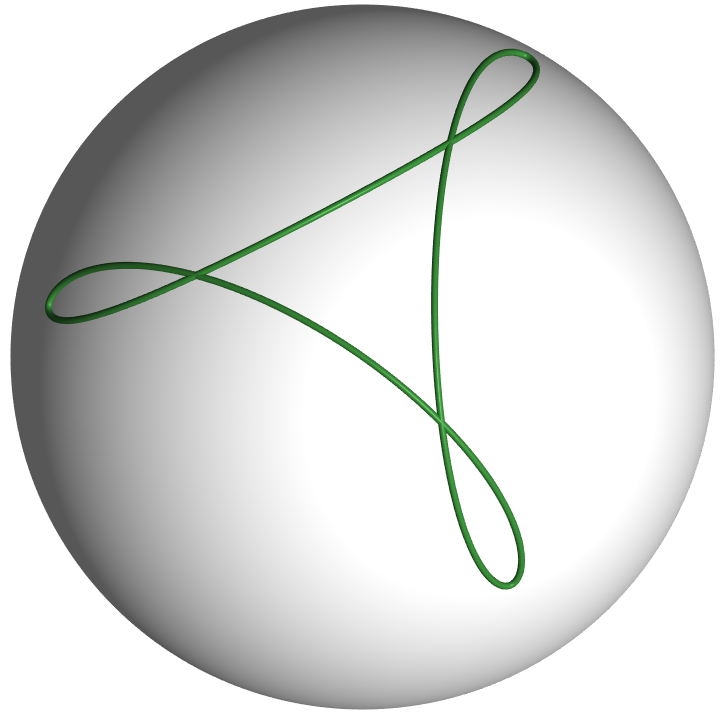}  	\quad & \quad
				\includegraphics[height=30mm]{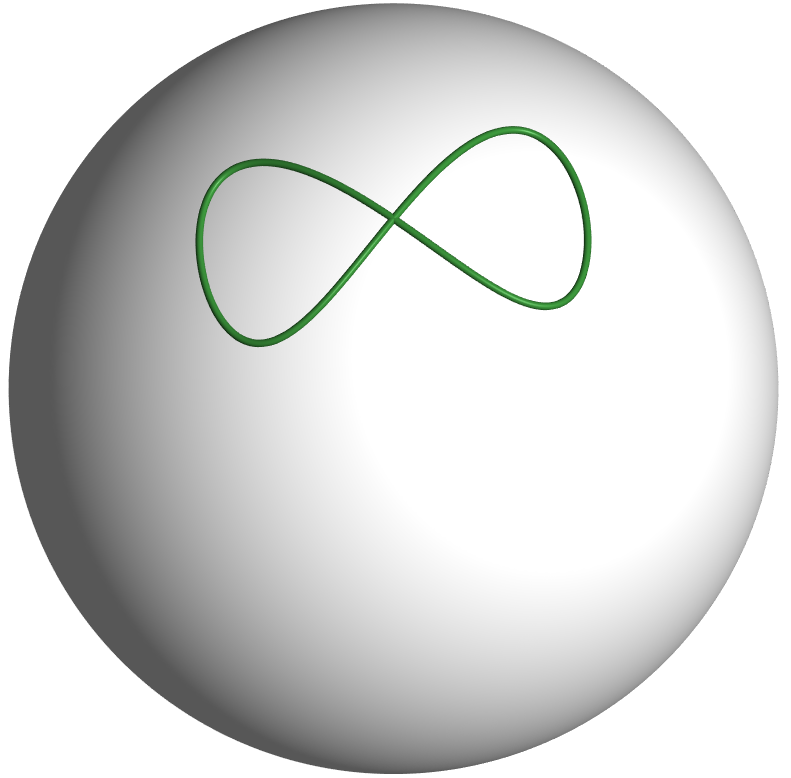}  	
\end{array}
	$
		\end{center}
	\caption{ Examples of binormal images of closed space curves.}
	\label{figcurves}
\end{figure}

Note that we only obtain curves $\bbr$ with non-vanishing torsion this way, and this does not
include all interesting cases.   For example if $\bbr$ takes values in 
 a sphere then the total torsion is zero.  Hence the torsion must change 
sign, unless $\bbr$ is a plane curve. It is possible to modify our discussion to include
these by including curves $\bbb$ that have cusps, but we omit this for the sake of simplicity.

\subsection{Explicit constructions}
Given a curve $\bbb$ in the $2$-sphere with the properties discussed above,
there are infinitely many closed space-curves $\bbr$ to choose from that have $\bbb$ as
the binormal field, so we would like to consider some natural choices.
For example closed curves with
non-vanishing curvature and \emph{constant} torsion are found in work by 
J. Weiner \cite{weiner1977},  and Bates and Melko \cite{batesmelko}, however 
to construct these a special choice of $\bbb$ is needed. Here we take $\bbb$ to be
arbitrary, subject to the conditions mentioned above.

Let $\bbb(t)$ be some parameterization of a regular closed curve in $\SSS^2$, 
with period $P$,
and the  length of one period $L_S$.
If $\bbb$ is the binormal field to a space curve $\bbr$ with unit tangent and normal fields $\bbt$ and $\bbn$ respectively,
then $\bbb^\prime$ is proportional to $\bbn$, so $\bbt = \bbn \times \bbb = (\bbb^\prime/|\bbb^\prime|) \times \bbb$. If $s$ is the arc-length parameter of $\bbr$, then
\[
\bbt = \frac{1}{\tau} \frac{\dd t}{\dd s}\frac{\dd \bbb}{\dd t} \times \bbb,
\]
and
\[
\bbr = \int \bbt \frac{\dd s}{\dd t} \dd t = \int \frac{1}{\tau} \frac{\dd \bbb}{\dd t} \times \bbb \, \dd t.
\]
Given a choice of $\bbb$, a curve $\bbr$ with binormal field $\bbb$ is
 determined by any choice of function $\sigma(t) := 1/\tau(t)$.
The closing condition for $\bbr$ is
\beq  \label{closingcond}
\int_0^P \sigma \frac{\dd \bbb}{\dd t} \times \bbb \, \dd t = 0,
\eeq
and clearly many choices for $\sigma$ will be available.  To find a solution, we can set
\[
\sigma(t) = \sum_{i=1}^n c_i B_i(t),
\]
where $B_i$ are some periodic basis functions, such as trigonometric functions or periodic B-splines.
  Then \eqref{closingcond} is:
\[
\sum_{i=1}^n c_i {\bf a}_i = 0, \quad \quad {\bf a}_i = \int_0^T B_i(t) \bbb(t) \times \bbb^\prime(t) \dd t.
\]
Thus, given a choice of basis functions, the solutions are found by linear algebra.  There can be many 
solutions for $c_i$, depending on the choice of basis functions.  To choose a natural solution,
we have used an optimization to minimize the elastic energy, $\int_0^L \kappa^2(s) \dd s$.
To write this quantity in terms of $\bbb$ and $\sigma$, let $s_T$ denote the arc-length parameter of
the curve $\bbt$ and compute
\[
\kappa =
 \tau \left| \bbb \times \frac{\dd ^2 \bbb}{\dd s_B^2}\right|
= \tau \left| \det \left( \bbb \, \frac{\dd \bbb}{\dd s_B}  \,\frac{\dd^2 \bbb}{\dd s_B^2} \right) \right|,
\]
and hence
\[
\mathcal{E}= \int_0^L \kappa^2 \dd s = \int_0^T \frac{1}{\sigma} \frac{\det(\bbb \bbb^\prime \bbb^{\prime \prime})^2}{|\bbb^\prime|^5} \dd t.
\]

\subsection{Concluding remarks}
We have used the method just described to produce the examples shown in Figure \ref{fig:tori}.  More 
examples can be found in the article \cite{bridges2017}.
If we scale $\bbb$  to have length $\pi$, and use a profile curve that passes through the origin and is
symmetric about the origin, then a M\"obius strip is obtained.  If the profile curve is a straight line
the surface is developable. The problem of constructing developable M\"obius strips has been considered,
for example, in \cite{wunderlich1962} and \cite{chicone}. The method of \cite{chicone} is, 
in essence, similar to that
presented here, only the authors assume that the curvature of  $\bbr$ is non-vanishing. In  \cite{rr1996},
developable M\"obius strips are characterized by a different method, using the so-called \emph{center}
geodesic curve instead of the spine curve.

If we take the spine curve of a developable M\"obius strip, and take the profile curve to be a symmetric
figure-$8$ curve, we obtain a Monge Klein bottle (see \cite{bridges2017}).   Matlab functions that allow
one to compute a Monge surface from a given spine curve and profile curve are currently 
available at  \href{http://geometry.compute.dtu.dk/software}{http://geometry.compute.dtu.dk/software}.


\end{document}